\documentclass[12pt]{amsart}

\usepackage{latexsym}
\usepackage{amsmath}
\usepackage{amssymb}
\usepackage{amsthm}
\usepackage{amscd}
\usepackage{color}
\usepackage{bm}

\newtheorem{theorem}{Theorem}[section]
\newtheorem{lemma}[theorem]{Lemma}

\newtheorem{proposition}[theorem]{Proposition}

\theoremstyle{definition}

\newtheorem{definition}[theorem]{Definition}
\newtheorem{example}[theorem]{Example}

\theoremstyle{remark}

\numberwithin{equation}{section}

\begin{document}

\title[Uniqueness of Hardy--H\'enon parabolic equations on Herz spaces]{Unconditional uniqueness of Hardy--H\'enon parabolic equations on Herz spaces}

\author{Naoya Hatano and Masahiro Ikeda}

\address[Naoya Hatano]{Graduate School of Information Science and Technology, The University of Osaka, 1-5, Yamadaoka, Suita-shi, Osaka 565-0871, Japan,}

\address[Masahiro Ikeda]{Graduate School of Information Science and Technology, The University of Osaka, 1-5, Yamadaoka, Suita-shi, Osaka 565-0871, Japan / Center for Advanced
Intelligence Project RIKEN, 1-4-1, Nihonbashi, Chuo-ku, Tokyo 103-0027, Japan.}

\email[Naoya Hatano]{n.hatano.chuo@gmail.com}

\email[Masahiro Ikeda]{ikeda@ist.osaka-u.ac.jp / masahiro.ikeda@a.riken.jp}

\begin{abstract}
In this paper, we introduce the unconditional uniqueness of solutions in Herz spaces for the Hardy--H\'enon parabolic equation, which is a semilinear heat equation with a power-type weight in the nonlinear term $|x|^\gamma|u|^{\alpha-1}u$.
It is expected that the power-type weight in the nonlinear term can be effectively handled within Herz spaces.
In fact, our result in Herz spaces $\dot{K}^s_{q,r}({\mathbb R}^n)$ relaxes the endpoint case $q=\alpha$ and the large interpolation exponent case $r\ge q$ compared to previous results.
\end{abstract}


\maketitle

\section{Introduction}

In this paper, we consider the Hardy–H\'enon parabolic equation
\begin{equation}\label{eq:HH}
\begin{cases}
\partial_t u-\Delta u
=
|x|^\gamma|u|^{\alpha-1}u,
& (t,x)\in(0,\infty)\times{\mathbb R}^n, \\
u(0)=u_0,
\end{cases}
\tag{HH}
\end{equation}
where $u_0$ is a given function, $\alpha>1$, and $\gamma\in{\mathbb R}$.
This equation is known as the Hardy--H\'enon parabolic equation, and the corresponding stationary problem was introduced and studied the rotating stellar systems by H\'enon in \cite{Henon73}.

In the case $\gamma=0$, which is the case of the Fujita-type semilinear heat equation, the existence and uniqueness of solutions in Lebesgue spaces to equation \eqref{eq:HH} have been studied by many authors (see, e.g., \cite{BrCa96,Weissler80,Weissler81} and the textbook \cite{GGS10,QuSo19}).
Moreover, the existence and uniqueness of solutions in Herz spaces were established by Drihem in \cite{Drihem23}.

When $-\min(2,n)<\gamma<0$, the well-posedness of equation \eqref{eq:HH} in Lebesgue spaces was shown by Ben Slimene, Tayachi, and Weissler in \cite{BTW17}, in Besov spaces by Chikami in \cite{Chikami19}, and in Sobolev spaces by Chikami, Ikeda and Taniguchi in \cite{CIT21}.
The unconditional uniqueness of solutions in Lorentz spaces was established by Tayachi in \cite{Tayachi20}.

When $\gamma>-\min(2,n)$, the existence of local and global bounded continuous positive solutions for continuous initial data was established by Wang in \cite{Wang93}.
The well-posedness of equation \eqref{eq:HH} in Lebesgue spaces with power weight was studied by Chikami, Ikeda, and Taniguchi in \cite{CIT22}, and in Lorentz spaces with power weight by Chikami, Ikeda, Taniguchi, and Tayachi in \cite{CITT25}.
The unconditional uniqueness of solutions in Lorentz spaces with power weight was established by Chikami, Ikeda, Taniguchi, and Tayachi in \cite{CITT24}.

Additionally, Yomgne \cite{Yomgne22} studied the equation \eqref{eq:HH} with the term $-\Delta u$ replaced by $(-\Delta)^mu$ for $m \in(0,1)\cup{\mathbb N}$, and established well-posedness in weak Lebesgue spaces as well as unconditional uniqueness in Lebesgue spaces.
Tsutsui \cite{Tsutsui25} considered the equation \eqref{eq:HH} with the nonlinear term $|x|^\gamma|u|^{\alpha-1}u$ replaced by $V|u|^{\alpha-1}u$ for a general weight function $V$, and proved the existence and unconditional uniqueness of solutions in weighted Lebesgue spaces.

By Duhamel's principle, the corresponding integral equation can be given by
\begin{equation}
u(t)
=
e^{t\Delta}u_0
+
\int_0^t
e^{(t-\tau)\Delta}
\left[
|\cdot|^\gamma|u(\tau)|^{\alpha-1}u(\tau)
\right]
\,{\rm d}\tau,
\end{equation}
where $\{e^{t\Delta}\}_{t>0}$ denotes the heat semigroup, defined by
\begin{equation}
e^{t\Delta}f(x)
\equiv
\frac1{\sqrt{(4\pi t)^n}}
\int_{{\mathbb R}^n}
\exp\left(-\frac{|x-y|^2}{4t}\right)
f(y)
\,{\rm d}y,
\quad
(t,x)\in(0,\infty)\times{\mathbb R}^n.
\end{equation}

The critical exponents $q_c$ and $Q_c$ are given by
\[
q_c
=
\frac{n(\alpha-1)}{2+\gamma},
\quad
Q_c
=
\frac{n\alpha}{n+\gamma}.
\]
In particular, the exponent $q_c$ is called the scale-critical exponent.
Indeed, the equation \eqref{eq:HH} is invariant under the scaling transformation
\[
u_\lambda(t,x)
=
\lambda^{\frac{2+\gamma}{\alpha-1}}
u(\lambda^2t,\lambda x),
\quad
\lambda>0,
\]
and the necessary and sufficient condition for the scaling invariance
\[
\|u_\lambda(0)\|_{\dot{K}^s_{q,r}}
\sim
\|u_0\|_{\dot{K}^s_{q,r}},
\quad
\lambda>0,
\]
is
\[
\frac sn+\frac1q=\frac1{q_c}.
\]
Accordingly, we say that the problem \eqref{eq:HH} is scale-critical if $s/n+1/q=1/q_c$, scale-subcritical if $s/n+1/q<1/q_c$, and scale-supercritical if $s/n+1/q>1/q_c$.
Additionally, the exponent $Q_c$ related to the well-definedness of the nonlinear term $|x|^\gamma|u|^{\alpha-1}u$ in $\dot{K}^s_{q,r}({\mathbb R}^n)$.
In fact, according to Proposition \ref{prop:well-defined} (2), below, we have $|x|^\gamma|u|^{\alpha-1}u\in L_{\rm loc}^1({\mathbb R}^n)$ for every $u\in\dot{K}^s_{q,r}({\mathbb R}^n)$ if and only if either
\[
\text{\lq\lq
$
\frac sn+\frac1q
<
\frac1{Q_c}
$
''}
\quad \text{or} \quad
\text{\lq\lq
$
\frac sn+\frac1q
=
\frac1{Q_c}
$
and
$r\le\alpha$
''}.
\]
To this end, we consider the following three critical cases.
\begin{itemize}
\item Double subcritical case:
$
s/n+1/q<\min(1/q_c,1/Q_c)
$.

\item Single critical case I:
$
s/n+1/q=1/Q_c<1/q_c
$.

\item Single critical case II:
$
s/n+1/q=1/q_c<1/Q_c
$.

\item Double critical case:
$
s/n+1/q=1/q_c=1/Q_c
$.
\end{itemize}

Throughout the main results of this paper, we always assume the following conditions
\begin{equation}\label{eq:Assum}
\begin{cases}
\gamma>-\min(2,n),
\quad
\alpha\ge\max\left(1,1+\dfrac\gamma n\right), \\
\dfrac\gamma{\alpha-1}\le s\le n,
\quad
\alpha\le q\le\infty,
\quad
0<r\le\infty
\end{cases}
\tag{$\ast$}
\end{equation}
and the exponent conditions in Proposition \ref{prop:well-defined}, below.

Under these assumptions, we obtain the following main results.

\begin{theorem}\label{main1}
If either
\begin{itemize}
\item[{\rm (i)}] $
s/n+1/q<\min(1/q_c,1/Q_c)
$,
\quad or

\item[{\rm (ii)}] $
s/n+1/q=1/Q_c<1/q_c
$
and
$r\le\alpha$,
\end{itemize}
then the solution of \eqref{eq:HH} in
$
L^\infty(
0,T\,;
\dot{K}^s_{q,r}({\mathbb R}^n)
)
$
is unique.
\end{theorem}

\begin{theorem}\label{main2}
Assume that
\[
n\ge3,
\quad
\frac\gamma{\alpha-1}<s,
\quad
\frac sn+\frac1q>0.
\]
If either
\begin{itemize}
\item[{\rm (i)}] $
s/n+1/q=1/q_c<1/Q_c
$,
$q<\infty$
and
$r=\infty$,
\quad or

\item[{\rm (ii)}] $
s/n+1/q=1/q_c=1/Q_c
$,
$
\begin{cases}
\text{
$q<\infty$
and
$r\le\alpha-1$
}, \\
\text{
$q\le\infty$
and
$r\le\min(1,\alpha-1)$
},
\end{cases}
$
\end{itemize}
then, for any initial data
$u_0\in\dot{\mathcal K}^s_{q,r}({\mathbb R}^n)$,
the solution of \eqref{eq:HH} in
$
C(
[0,T]\,;
\dot{\mathcal K}^s_{q,r}({\mathbb R}^n)
)
$
is unique, where
$
\dot{\mathcal K}^s_{q,r}({\mathbb R}^n)
$
is the closure of the space of all smooth functions with compact support
$
C_{\rm c}^\infty({\mathbb R}^n)
$
in $\dot{K}^s_{q,r}({\mathbb R}^n)$.
\end{theorem}

Note that the necessary and sufficient condition for $C_{\rm c}^\infty({\mathbb R}^n)$ to be dense in $\dot{K}^s_{q,r}({\mathbb R}^n)$ is $q,r<\infty$ (see Proposition \ref{prop:dense}).

In the cases of the Lorentz spaces $L_s^{q,r}({\mathbb R}^n)$, the usual endpoint restrictions $q=\alpha$ or $q=\infty$ can be relaxed by working instead with the Herz spaces $\dot{K}^s_{q,r}({\mathbb R}^n)$.

As a key proposition for the proof of Theorem 1.2, which is the scale critical case, we use the Meyer's inequality (see Proposition \ref{prop:Meyer-Herz} below).
We use the following extention of the Meyer's inequality, where the original Meyer's inequality can be seen in \cite{Terraneo02}.
It is known that if
\[
\frac1q=\frac1p-\frac2n,
\]
then
\[
\left\|
\int_0^t
e^{(t-\tau)\Delta}f(\tau)
\,{\rm d}\tau
\right\|_{L^{q,\infty}}
\lesssim
\sup_{0<\tau<t}
\|f(\tau)\|_{L^{p,\infty}},
\]
where $L^{p,\infty}({\mathbb R}^n)$ and $L^{q,\infty}({\mathbb R}^n)$ are weak Lebesgue spaces.
This inequality is called Meyer's inequality, and Meyer \cite{Meyer97} showed the uniqueness and existence of the solution to the Navier--Stokes equation in
$
L^{3,\infty}_\sigma({\mathbb R}^3)
=
\{
f\in L^{3,\infty}({\mathbb R}^3)
\,:\,
{\rm div}f=0
\}
$
using this inequality.
After that Terraneo \cite{Terraneo02} applied Meyer's inequality to the Fujita-type semilinear heat equation, and investigate the non-uniquenesss of the solutions.
Additionally, Chikami et al. \cite{CITT24} and Tsutsui \cite{Tsutsui11} gave the generalization for Meyer's inequality on the weak Lebesgue spaces with power weight and weak Herz spaces, respectively.

We organize the remainder of the paper as follows:
In Section \ref{s:Herz}, we give the precise definition of Herz spaces and their basic properties.
In Section \ref{s:LE}, we extend the classical heat semigroup estimates and Meyer's inequalities to Herz spaces.
In Section \ref{s:proof}, we provides the proof of the main theorem.
In Section \ref{s:rem}, we compare Herz spaces and Lorentz spaces with power weights, in order to contrast our results with previous ones.

\section{Herz spaces}\label{s:Herz}

In this section, in order to suitably control the nonlinear term in equation \eqref{eq:HH}, we give the precise definition of Herz spaces and some basic properties of these function spaces.

\begin{definition}
Let $s\in{\mathbb R}$ and $0<q,r\le\infty$.
Define the (homogeneous) Herz space $\dot{K}^s_{q,r}({\mathbb R}^n)$ by the space of all measurable functions $f$ with the finite quasi-norm
\[
\|f\|_{\dot{K}^s_{q,r}}
\equiv
\begin{cases}
\displaystyle
\left(
\sum_{j\in{\mathbb Z}}
\left[
2^{js}
\|f\chi_{A_j}\|_{L^q}
\right]^r
\right)^{\frac1r},
& r<\infty, \\
\displaystyle
\sup_{j\in{\mathbb Z}}
2^{js}
\|f\chi_{A_j}\|_{L^q},
& r=\infty,
\end{cases}
\]
where
$
A_j
=
\{
x\in{\mathbb R}^n
\,:\,
2^{j-1}\le|x|<2^j
\}
$.
\end{definition}

Related to the assumption of the main results, we give the following proposition.

\begin{proposition}\label{prop:well-defined}
Let $s\in{\mathbb R}$ and $0<q,r\le\infty$.
Then the following assertions hold.
\begin{itemize}
\item[{\rm (1)}] The inclusion $
\dot{K}^s_{q,r}({\mathbb R}^n)
\supset
C_{\rm c}^\infty({\mathbb R}^n)
$
holds if and only if either
\begin{itemize}
\item[{\rm (i)}] $s/n+1/q>0$
\quad or

\item[{\rm (ii)}] $s/n+1/q=0$
and
$r=\infty$.
\end{itemize}

\item[{\rm (2)}] The inclusion $
\dot{K}^s_{q,r}({\mathbb R}^n)
\subset
L_{\rm loc}^1({\mathbb R}^n)
$
holds if and only if either
\begin{itemize}
\item[{\rm (i)}] $s/n+1/q<1$
\quad or

\item[{\rm (ii)}] $s/n+1/q=1$
and
$r\le1$.
\end{itemize}
\end{itemize}
\end{proposition}

Thus, in what follows, we work with the Herz spaces $\dot{K}^s_{q,r}({\mathbb R}^n)$ under the exponent conditions specified in this proposition.
Since $|x|\sim2^j$ on $A_j$, we have
$
\||\cdot|^\beta f\|_{\dot{K}^s_{q,r}}
\sim
\|f\|_{\dot{K}^{s+\beta}_{q,r}}
$.
Hence, the use of Herz spaces always enables us to effectively handle the power weight $|x|^\gamma$ appearing in the nonlinear term of equation \eqref{eq:HH}.

Here and in what follows, we present some basic properties of Herz spaces.

\begin{proposition}\label{prop:Herz-fan}
Let the parameters $s,s_1,s_2\in{\mathbb R}$ and $0<q,q_1,q_2,r,r_1,r_2\le\infty$.
\begin{itemize}
\item[{\rm (1)}] If
\[
s=s_1+s_2,
\quad
\frac1q=\frac1{q_1}+\frac1{q_2},
\quad
\frac1r=\frac1{r_1}+\frac1{r_2},
\]
then,
\[
\|fg\|_{\dot{K}^s_{q,r}}
\le
\|f\|_{\dot{K}^{s_1}_{q_1,r_1}}
\|g\|_{\dot{K}^{s_2}_{q_2,r_2}}.
\]

\item[{\rm (2)}] For $\eta>0$,
\[
\||f|^\eta\|_{\dot{K}^s_{q,r}}
=
\|f\|_{\dot{K}^{s/\eta}_{q\eta,r\eta}}^\eta.
\]

\item[{\rm (3)}] For $\lambda>0$,
\[
\|f(\lambda\cdot)\|_{\dot{K}^s_{q,r}}
\sim
\lambda^{-s+\frac nq}
\|f\|_{\dot{K}^s_{q,r}}.
\]

\item[{\rm (4)}] If $0<r_1\le r_2\le\infty$, then,
\[
\dot{K}^s_{q,r_1}({\mathbb R}^n)
\hookrightarrow
\dot{K}^s_{q,r_2}({\mathbb R}^n).
\]

\item[{\rm (5)}] If
\[
\frac{s_1}n+\frac1{q_1}
=
\frac{s_2}n+\frac1{q_2},
\quad
q_1\ge q_2,
\]
then,
\[
\dot{K}^{s_1}_{q_1,r}({\mathbb R}^n)
\hookrightarrow
\dot{K}^{s_2}_{q_2,r}({\mathbb R}^n).
\]

\item[{\rm (6)}] $
\|f\|_{\dot{K}^s_{q,q}}
\sim
\||\cdot|^sf\|_{L^q}
$.

\item[{\rm (7)}] \cite{HeYa97}
If $1\le q,r<\infty$, then,
\[
\left(
\dot{K}^s_{q,r}({\mathbb R}^n)
\right)^\ast
\cong
\dot{K}^{-s}_{q',r'}({\mathbb R}^n).
\]
\end{itemize}
\end{proposition}

We provide the proofs of the following lemma and propositions in the appendix, which are new as far as we know.

\begin{lemma}\label{lem:Herz-ball}
Let $s<0$ and $0<q,r\le\infty$.
Then,
\[
\|f\|_{\dot{K}^s_{q,r}}
\sim
\begin{cases}
\displaystyle
\left(
\sum_{j\in{\mathbb Z}}
\left[
2^{js}
\|f\chi_{B(2^j)}\|_{L^q}
\right]^r
\right)^{\frac1r},
& r<\infty, \\
\displaystyle
\sup_{j\in{\mathbb Z}}
2^{js}
\|f\chi_{B(2^j)}\|_{L^q},
& r=\infty,
\end{cases}
\]
where $B(2^j)=B(0,2^j)$.
\end{lemma}

\begin{proposition}\label{prop:dense}
Under the exponent condition in Proposition {\rm \ref{prop:well-defined} (1)}, the necessary and sufficient condition for $C_{\rm c}^\infty({\mathbb R}^n)$ to be dense in $\dot{K}^s_{q,r}({\mathbb R}^n)$ is $0<q,r<\infty$.
\end{proposition}

\begin{proposition}\label{prop:Gauss}
The necessary and sufficient condition for
$
e^{-|\cdot|^2}\in
\dot{K}^s_{q,r}({\mathbb R}^n)
$
is  the exponent condition in Proposition {\rm \ref{prop:well-defined} (1)}.
\end{proposition}

Additionally, we use the real interpolation spaces for the Herz spaces, defined via the $K$-functional.

\begin{definition}
Let $X_0=(X_0,\|\cdot\|_{X_0})$ and $X_1=(X_1,\|\cdot\|_{X_1})$ be quasi-Banach spaces of measurable functions, and let $0<\theta<1$ and $0<r\le\infty$.
Define the $K$ functional as
\[
K(t,f)
\equiv
\inf_{
\substack{f=f_0+f_1, \\ (f_0,f_1)\in X_0\times X_1}
}
\left(
\|f_0\|_{X_0}
+
t\|f_1\|_{X_1}
\right),
\quad t>0,
\]
for all $f\in X_0+X_1$, and the real interpolation space $(X_0,X_1)_{\theta,r}$ as the space of all measurable functions $f$ with the finite quasi-norm
\[
\|f\|_{(X_0,X_1)_{\theta,r}}
\equiv
\begin{cases}
\displaystyle
\left(
\int_0^\infty
\left[
t^{-\theta}
K(t,f)
\right]^r
\,\frac{{\rm d}t}t
\right)^{\frac1r},
& r<\infty, \\
\displaystyle
\sup_{t>0}
t^{-\theta}
K(t,f),
& r=\infty.
\end{cases}
\]
\end{definition}

\begin{theorem}\label{thm:real-Herz}
Let $0<\theta<1$, $s,s_0,s_1\in{\mathbb R}$ and $0<q,r,r_0,r_1\le\infty$.
If
\[
s_0\ne s_1,
\quad
s=(1-\theta)s_0+\theta s_1,
\]
then,
\[
(
\dot{K}^{s_0}_{q,r_0}({\mathbb R}^n),
\dot{K}^{s_1}_{q,r_1}({\mathbb R}^n)
)_{\theta,r}
\cong
\dot{K}^s_{q,r}({\mathbb R}^n).
\]
\end{theorem}

The case $1<q,r,r_0,r_1<\infty$ in this theorem was treated by Ho \cite{Ho18}.
Meanwhile, according to Lemma \ref{lem:Herz-ball}, the case $s<0$ corresponds to local Morrey-type spaces, and the real interpolation spaces for these function spaces were studied by Burenkov and Nursultanov \cite{BuNu09}.
In this paper, following the method of Ho, we provide a proof of the remaining cases in Appendix \ref{App:interpolation}.

\section{Linear estimates}\label{s:LE}

It is well known that the $L^p({\mathbb R}^n)$-$L^q({\mathbb R}^n)$ estimate
\[
\|e^{t\Delta}f\|_{L^q}
\le
t^{-\frac n2\left(\frac1p-\frac1q\right)}
\|f\|_{L^p}
\]
holds for all $f\in L^p({\mathbb R}^n)$ when $1\le p\le q\le\infty$.
In the first half of this section, we consider an extension to Herz spaces.
In the second half of this section, we present Meyer's inequality on Herz spaces, which is used to estimate the nonlinear term in the scale critical case.

\subsection{Extension of $L^p({\mathbb R}^n)$-$L^q({\mathbb R}^n)$ estimates}

The extension to using the Herz spaces are given as follows.

\begin{proposition}\label{prop:LE-Herz2}
Let
$\mu,\nu\in{\mathbb R}$,
$1\le p,q\le\infty$,
and
$0<r,r_0\le\infty$.
Assume that
\[
\mu\ge\nu,
\quad
0
\le
\frac\nu n+\frac1q
\le
\frac\mu n+\frac1p
\le1,
\quad
\begin{cases}
r_0=\infty &
\text{if \;
$\displaystyle
\frac\mu n+\frac1p=0
$}, \vspace{5pt} \\
r=\infty &
\text{if \;
$\displaystyle
\frac\nu n+\frac1q=0
$}.
\end{cases}
\]
If either
\begin{itemize}
\item[{\rm (1)}] $\mu>\nu$,
\quad
$p,q<\infty$,
\quad
$r_0=\infty$,
\quad
$\displaystyle
\frac\nu n+\frac1q
<
\frac\mu n+\frac1p
<1
$,
\quad or

\item[{\rm (2)}] $p,q<\infty$,
\quad
$r_0\le r$,
\quad
$\displaystyle
\frac\mu n+\frac1p
<1
$,
\quad or

\item[{\rm (3)}] $r_0\le\min(1,r)$,
\end{itemize}
then,
\[
\|e^{t\Delta}f\|_{\dot{K}^\nu_{q,r}}
\lesssim
t^{
-\frac n2
\left[
\left(
\frac\mu n+\frac1p
\right)
-
\left(
\frac\nu n+\frac1q
\right)
\right]
}
\|f\|_{\dot{K}^\mu_{p,r_0}}.
\]
\end{proposition}

The cases (1) and (2) were treated by Drihem \cite{Drihem23}.
In this paper, we rewrite the entire proof of this theorem, following the approach in \cite{Tsutsui11}, and we include the additional case (3).

\begin{proof}
By Proposition \ref{prop:Herz-fan} (3), it suffices to show that
\[
\|G\ast f\|_{\dot{K}^\nu_{q,r}}
\lesssim
\|f\|_{\dot{K}^\mu_{p,r_0}},
\]
where $G(x)=e^{-|x|^2}$.

\noindent
\underline{\bf Step. 1}.
Assume that $\mu\le0$, and take $\tilde{\mu}$, $\tilde{p}$, $\alpha$ and $M$ by
\[
\nu=\mu+\tilde{\mu},
\quad
\frac1q+1
=
\frac1p+\frac1{\tilde{p}},
\quad
\alpha-\mu+\frac n{p'}\ge0,
\quad
M-\alpha+\mu+\frac n{p'}
>
\nu+\frac nq.
\]
Note that by $G\in{\mathcal S}({\mathbb R}^n)$,
$
G(x)
\lesssim
\min(1,|x|^{-M})
$.
Let $G_\alpha(x)=|x|^\alpha G(x)$.

For each $j\in{\mathbb Z}$, we decompose
\begin{align*}
G_\alpha\ast f(x)\chi_{A_j}(x)
&=
\sum_{k\in{\mathbb Z}}
\int_{A_k}G_\alpha(x-y)f(y)\,{\rm d}y
\;\chi_{A_j}(x)\\
&=
\sum_{k\in{\mathbb Z}}
\int_{{\mathbb R}^n}
G_\alpha(x-y)\chi_{A_j-A_k}(x-y)
\cdot
f(y)\chi_{A_k}(y)
\,{\rm d}y\\
&=
\sum_{k\in{\mathbb Z}}
[G_\alpha\chi_{A_j-A_k}]\ast[f\chi_{A_k}](x)\\
&=
\sum_{k<j-1}
+
\sum_{k=j-1}^{j+1}
+
\sum_{k>j+1}
=:
{\rm I}_j(x)+{\rm II}_j(x)+{\rm III}_j(x),
\end{align*}
where
$
A_j-A_k
\equiv
\{x-y\,:\,(x,y)\in A_j\times A_k\}
$.
Here, by the simple geometric consideration,
\[
A_j-A_k
\subset
\begin{cases}
A_{j-1}\cup A_j\cup A_{j+1}, & \text{if}\; k<j-1, \\
B(2^{j+2}), & \text{if}\; k=j-1,j,j+1, \\
A_{k-1}\cup A_k\cup A_{k+1}, & \text{if}\; k>j+1.
\end{cases}
\]
Then, setting
$
{\rm I}(x)=\{{\rm I}_j(x)\}_{j\in{\mathbb Z}}
$,
$
{\rm II}(x)=\{{\rm II}_j(x)\}_{j\in{\mathbb Z}}
$
and
$
{\rm III}(x)=\{{\rm III}_j(x)\}_{j\in{\mathbb Z}}
$,
and using the vector-valued norm
\[
\|F\|_{\ell_\nu^r(L^q)}
=
\left(
\sum_{j\in{\mathbb Z}}
\left[
2^{j\nu}
\|f_j\|_{L^q}
\right]^r
\right)^{\frac1r},
\]
we have
\[
\|G_\alpha\ast f\|_{\dot{K}^\nu_{q,r}}
\lesssim
\|{\rm I}\|_{\ell_\nu^r(L^q)}
+
\|{\rm II}\|_{\ell_\nu^r(L^q)}
+
\|{\rm III}\|_{\ell_\nu^r(L^q)}.
\]

On the estimate of ${\rm I}$,
\begin{align*}
\|{\rm I}\|_{\ell_\nu^r(L^q)}^r
&\le
\sum_{j\in{\mathbb Z}}
\left[
2^{j\nu}
\sum_{k<j-1}
\|
[G_\alpha\chi_{A_{j-1}\cup A_j\cup A_{j+1}}]
\ast
[|f|\chi_{A_k}]
\|_{L^q}
\right]^r\\
&\le
\sum_{j\in{\mathbb Z}}
\left[
2^{j\nu}
\sum_{k<j-1}
\|G_\alpha\chi_{A_{j-1}\cup A_j\cup A_{j+1}}\|_{L^{\tilde{p}}}
\|f\chi_{A_k}\|_{L^p}
\right]^r\\
&\lesssim
\sum_{j\in{\mathbb Z}}
\left[
2^{j\tilde{\mu}}
\|G_\alpha\chi_{A_{j-1}\cup A_j\cup A_{j+1}}\|_{L^{\tilde{p}}}
\|f\|_{\dot{K}^\mu_{p,r_0}}
\right]^r
\lesssim
\|G_\alpha\|_{\dot{K}^{\tilde{\mu}}_{\tilde{p},r}}^r
\|f\|_{\dot{K}^\mu_{p,r_0}}^r,
\end{align*}
where
\[
\begin{cases}
r_0=\infty, & \text{if $\mu<0$}, \\
r_0=1, & \text{if $\mu\le0$}.
\end{cases}
\]

On the estimate of ${\rm II}$,
\begin{align*}
\|{\rm II}\|_{\ell_\nu^r(L^q)}^r
&\le
\sum_{j\in{\mathbb Z}}
\left[
2^{j\nu}
\sum_{k=j-1}^{j+1}
\|[G_\alpha\chi_{B(2^{j+2})}]\ast[|f|\chi_{A_k}]\|_{L^q}
\right]^r\\
&\lesssim
\sum_{j\in{\mathbb Z}}
\left[
2^{j\nu}
\|G_\alpha\chi_{B(2^{j+2})}\|_{L^{\tilde{p}}}
\|f\chi_{A_{j-1}\cup A_j\cup A_{j+1}}\|_{L^p}
\right]^r\\
&\lesssim
\sum_{j\in{\mathbb Z}}
\left[
2^{j\tilde{\mu}}
\|G_\alpha\chi_{B(2^{j+2})}\|_{L^{\tilde{p}}}
\|f\|_{\dot{K}^\mu_{p,r_0}}
\right]^r,
\end{align*}
and then, by Lemma \ref{lem:Herz-ball},
\[
\|{\rm II}\|_{\ell_\nu^r(L^q)}
\lesssim
\begin{cases}
\|G_\alpha\|_{\dot{K}^{\tilde{\mu}}_{\tilde{p},r}}
\|f\|_{\dot{K}^\mu_{p,\infty}},
& \text{if $\tilde{\mu}<0$}, \\
\|G_\alpha\|_{L^{\tilde{p}}}
\|f\|_{\dot{K}^\mu_{p,r}},
& \text{if $\tilde{\mu}\le0$}.
\end{cases}
\]

On the estimate of ${\rm III}$,
\begin{align*}
\|{\rm III}\|_{\ell_\nu^r(L^q)}^r
&\le
\sum_{j\in{\mathbb Z}}
\left[
2^{j\nu}
\sum_{k>j+1}
\|
(
[G_\alpha\chi_{A_{k-1}\cup A_k\cup A_{k+1}}]
\ast
[|f|\chi_{A_k}]
)\chi_{A_j}
\|_{L^q}
\right]^r\\
&\lesssim
\sum_{j\in{\mathbb Z}}
\left[
2^{j\left(\nu+\frac nq\right)}
\sum_{k>j+1}
2^{k\alpha}
\min(1,2^{-kM})
\|f\chi_{A_k}\|_{L^1}
\right]^r
\\
&\le
\sum_{j\in{\mathbb Z}}
\left[
2^{j\left(\nu+\frac nq\right)}
\sum_{k>j+1}
2^{k\left(\alpha-\mu+\frac n{p'}\right)}
\min(1,2^{-kM})
\right]^r
\|f\|_{\dot{K}^\mu_{p,\infty}}^r
\sim
\|f\|_{\dot{K}^\mu_{p,\infty}}^r.
\end{align*}

\noindent
\underline{\bf Step. 2}.
Assume that $\nu\le0\le\mu$.
Choosing a suitable parameter $\mu_0$ and $p_0$ satisfying
\[
\nu\le\mu_0\le0,
\quad
p_0\le p,
\quad
\frac{\mu_0}n+\frac1{p_0}=\frac\mu n+\frac1p,
\]
we have
\[
\|G_\alpha\ast f\|_{\dot{K}^\nu_{q,r}}
\lesssim
\|f\|_{\dot{K}^{\mu_0}_{p_0,r_0}}.
\]
Then, by the embedding
$
\dot{K}^{\mu_0}_{p_0,r_0}({\mathbb R}^n)
\hookleftarrow
\dot{K}^\mu_{p,r_0}({\mathbb R}^n)
$
(see Proposition \ref{prop:Herz-fan} (6)),
\[
\|G_\alpha\ast f\|_{\dot{K}^\nu_{q,r}}
\lesssim
\|f\|_{\dot{K}^\mu_{p,r_0}}.
\]

\noindent
\underline{\bf Step. 3}.
Assume that $0\le\nu\le\mu$.
Let $\nu_0$ satisfy $\nu_0\le0$ and $\nu-\nu_0\ge0$.
Then, using the estimate
\[
|x|^{\nu-\nu_0}
\lesssim
|x-y|^{\nu-\nu_0}+|y|^{\nu-\nu_0},
\]
we can split
\begin{align*}
\|G\ast f\|_{\dot{K}^\nu_{q,r}}
&\sim
\|
|\cdot|^{\nu-\nu_0}
(G\ast f)
\|_{\dot{K}^{\nu_0}_{q,r}}\\
&\lesssim
\|
[|\cdot|^{\nu-\nu_0}G]\ast|f|
\|_{\dot{K}^{\nu_0}_{q,r}}
+
\|
G\ast[|\cdot|^{\nu-\nu_0}|f|]
\|_{\dot{K}^{\nu_0}_{q,r}}\\
&=:
{\rm I}+{\rm II}.
\end{align*}

On the estimate of ${\rm I}$, taking $\mu_1\le0$, $\tilde{\mu}_1\le0$ and $1\le p_1,\tilde{p}_1\le\infty$ such that
\[
\nu_0=\mu_1+\tilde{\mu}_1,
\quad
\frac1q+1=\frac1{p_1}+\frac1{\tilde{p}_1},
\quad
\frac{\mu_1}n+\frac1{p_1}=\frac\mu n+\frac1p,
\]
we can use the estimates obtained in Steps 1 and 2, and then, by Proposition \ref{prop:Herz-fan} (5),
\begin{align*}
{\rm I}
\lesssim
\|f\|_{\dot{K}^{\mu_1}_{p_1,r_0}}
\lesssim
\|f\|_{\dot{K}^\mu_{p,r_0}}.
\end{align*}

On the estimate of ${\rm I}$, taking $\mu_2\le0$, $\tilde{\mu}_2\le0$ and $1\le p_2,\tilde{p}_2\le\infty$ such that
\[
\nu_0=\mu_2+\tilde{\mu}_2,
\quad
\frac1q+1=\frac1{p_2}+\frac1{\tilde{p}_2},
\]
\[
p_2=p,
\quad
\mu_2+\nu-\nu_0=\mu,
\]
we can use the estimates obtained in Steps 1 and 2, and then,
\begin{align*}
{\rm II}
\lesssim
\|
|\cdot|^{\nu-\nu_0}f
\|_{\dot{K}^{\mu_2}_{p,r_0}}
\sim
\|f\|_{\dot{K}^\mu_{p,r_0}},
\end{align*}
as desired.
\end{proof}

\subsection{Meyer's inequality}

In this subsection, we gave the generalization for Meyer's inequality on the Herz spaces.

\begin{proposition}\label{prop:Meyer-Herz}
Let $n\ge3$, $\mu,\nu\in{\mathbb R}$, $1\le p,q\le\infty$ and $0<r\le\infty$.
Assume that
\[
\mu>\nu,
\quad
0<
\frac\nu n+\frac1q
<
\frac\mu n+\frac1p
\le1,
\]
\[
\frac\nu n+\frac1q
=
\left(
\frac\mu n+\frac1p
\right)
-
\frac2n.
\]
If either
\begin{itemize}
\item[{\rm (1)}] $p,q<\infty$,
\quad
$\displaystyle
\frac\mu n+\frac1p<1
$,
\quad or

\item[{\rm (2)}] $p,q\le\infty$,
\quad
$\displaystyle
\frac\mu n+\frac1p\le1
$,
\quad
$r\le1$,
\end{itemize}
then, for all $t>0$,
\[
\left\|
\int_0^t
e^{(t-\tau)\Delta}f(\tau)
\,{\rm d}\tau
\right\|_{\dot{K}^\nu_{q,\infty}}
\lesssim
\sup_{0<\tau<t}
\|f(\tau)\|_{\dot{K}^\mu_{p,r}}.
\]
\end{proposition}

\begin{proof}
According to the assumption
\[
\frac\nu n+\frac1q
=
\left(
\frac\mu n+\frac1p
\right)
-
\frac2n,
\]
we can choose $\nu_0,\nu_1$ such that
\begin{align*}
0<
\theta
&:=
1
-
\frac n2
\left[
\left(
\frac\mu n+\frac1p
\right)
-
\left(
\frac{\nu_0}n+\frac1q
\right)
\right]\\
&=
2
-
\frac n2
\left[
\left(
\frac\mu n+\frac1p
\right)
-
\left(
\frac{\nu_1}n+\frac1q
\right)
\right]
<
1.
\end{align*}
And then,
\begin{align*}
&(1-\theta)\nu_0+\theta \nu_1\\
&\qquad =
(1-\theta)
\cdot n
\left[
\left(
\frac\mu n+\frac1p
\right)
-
\frac1q-\frac{2-2\theta}n
\right]
+
\theta
\cdot n
\left[
\left(
\frac\mu n+\frac1p
\right)
-
\frac1q-\frac{4-2\theta}n
\right]\\
&\qquad =
\mu+\frac np-\frac nq-2
=
\nu.
\end{align*}
If we reselect $\nu_0,\nu_1$ by $\nu$ as needed, then, for each $i=0,1$,
\[
\mu>\nu_i,
\quad
0<
\frac{\nu_i}n+\frac1q
<
\frac\mu n+\frac1p
\le1.
\]
Thus, we can use Proposition \ref{prop:LE-Herz2}, and then,
\[
\|e^{t\Delta}h\|_{\dot{K}^{\nu_i}_{q,\infty}}
\lesssim
t^{
-\frac n2
\left[
\left(
\frac\mu n+\frac1p
\right)
-
\left(
\frac{\nu_i}n+\frac1q
\right)
\right]
}
\|h\|_{\dot{K}^\mu_{p,r}}
\]
for all $h\in\dot{K}^\mu_{p,r}({\mathbb R}^n)$.

We decompose
\begin{align*}
g(t,x)
&:=
\int_0^t
e^{(t-\tau)\Delta}f(\tau,x)
\,{\rm d}\tau\\
&=
\int_0^{\tilde{t}}
e^{(t-\tau)\Delta}f(\tau,x)
\,{\rm d}\tau
+
\int_{\tilde{t}}^t
e^{(t-\tau)\Delta}f(\tau,x)
\,{\rm d}\tau\\
&=:
g_{\tilde{t}}(x)+g^{\tilde{t}}(x).
\end{align*}
We estimate
\begin{align*}
\|g_{\tilde{t}}(t)\|_{\dot{K}^{\nu_0}_{q,\infty}}
&\lesssim
\int_0^{\tilde{t}}
(t-\tau)^{
-\frac n2
\left[
\left(
\frac\mu n+\frac1p
\right)
-
\left(
\frac{\nu_0}n+\frac1q
\right)
\right]
}
\|f(\tau)\|_{\dot{K}^\mu_{p,r}}
\,{\rm d}\tau\\
&\lesssim
(t-\tilde{t})^{
-\frac n2
\left[
\left(
\frac\mu n+\frac1p
\right)
-
\left(
\frac{\nu_0}n+\frac1q
\right)
\right]
+1
}
\sup_{0<\tau<t}
\|f(\tau)\|_{\dot{K}^\mu_{p,r}}\\
&=
(t-\tilde{t})^\theta
\sup_{0<\tau<t}
\|f(\tau)\|_{\dot{K}^\mu_{p,r}},
\end{align*}
and
\begin{align*}
\|g^{\tilde{t}}(t)\|_{\dot{K}^{\nu_1}_{q,\infty}}
&\lesssim
\int_{\tilde{t}}^t
(t-\tilde{t})^{
-\frac n2
\left[
\left(
\frac\mu n+\frac1p
\right)
-
\left(
\frac{\nu_1}n+\frac1q
\right)
\right]
}
\|f(\tau)\|_{\dot{K}^\mu_{p,r}}
\,{\rm d}\tau\\
&\lesssim
(t-\tau)^{
-\frac n2
\left[
\left(
\frac\mu n+\frac1p
\right)
-
\left(
\frac{\nu_1}n+\frac1q
\right)
\right]
+1
}
\sup_{0<\tau<t}
\|f(\tau)\|_{\dot{K}^\mu_{p,r}}\\
&=
(t-\tilde{t})^{\theta-1}
\sup_{0<\tau<t}
\|f(\tau)\|_{\dot{K}^\mu_{p,r}}.
\end{align*}
Combining the estimates for $g_{\tilde{t}}(t,x)$ and $g^{\tilde{t}}(t,x)$, we have
\begin{align*}
\|g(t)\|_{
(
\dot{K}^{\nu_0}_{q,\infty},
\dot{K}^{\nu_1}_{q,\infty}
)_{\theta,\infty}
}
&\le
\sup_{0<\tilde{t}<t}
(t-\tilde{t})^{-\theta}
\left(
\|g_{\tilde{t}}(t)\|_{\dot{K}^{\nu_0}_{q,\infty}}
+
(t-\tilde{t})
\|g^{\tilde{t}}(t)\|_{\dot{K}^{\nu_1}_{q,\infty}}
\right)\\
&\lesssim
\sup_{0<\tau<t}
\|f(\tau)\|_{\dot{K}^\mu_{p,r}}.
\end{align*}
Consequently, by Theorem \ref{thm:real-Herz}, we obtain
\[
\|g(t)\|_{\dot{K}^\nu_{q,\infty}}
\lesssim
\sup_{0<\tau<t}
\|f(\tau)\|_{\dot{K}^\mu_{p,r}}.
\]
\end{proof}

\section{Proof of the main theorems}\label{s:proof}

In this section, we give proofs of Theorems \ref{main1} and \ref{main2}.
To prove these theorems, we set $\sigma\equiv\alpha s-\gamma$, and use Propositions \ref{prop:LE-Herz2} and \ref{prop:Meyer-Herz} with $(\mu,\nu,p)=(\sigma,s,q/\alpha)$.
Then, remark that
\[
\frac sn+\frac1q
\le
\frac1{q_c}
\quad \text{if and only if} \quad
\frac sn+\frac1q
\ge
\left(
\frac\sigma n+\frac\alpha q
\right)
-
\frac2n,
\]
and
\[
\frac sn+\frac1q
\le
\frac1{Q_c}
\quad \text{if and only if} \quad
\frac\sigma n+\frac\alpha q
\le
1
\]
hold.

\subsection{Proof of Theorem \ref{main1}}

Let
$
u_1,u_2
\in
L^\infty(
0,T
\,;\,
\dot{K}^s_{q,r}({\mathbb R}^n)
)
$
be solutions of \eqref{eq:HH} having the same initial data.
By Proposition \ref{prop:LE-Herz2} and Proposition \ref{prop:Herz-fan} (2),
\begin{align*}
&\|u_1(t)-u_2(t)\|_{\dot{K}^s_{q,r}}
=
\|Nu_1(t)-Nu_2(t)\|_{\dot{K}^s_{q,r}}\\
&\lesssim
\int_0^t
(t-\tau)^{
-\frac n2
\left[
\left(
\frac\sigma n+\frac\alpha q
\right)
-
\left(
\frac sn+\frac1q
\right)
\right]
}
\|
|\cdot|^\gamma
\left(
|u_1(\tau)|^{\alpha-1}u_1(\tau)
-
|u_2(\tau)|^{\alpha-1}u_2(\tau)
\right)
\|_{\dot{K}^\sigma_{q/\alpha,r_0}}
\,{\rm d}\tau\\
&\lesssim
\int_0^t
(t-\tau)^{\delta-1}
\|
(|u_1(\tau)|^{\alpha-1}+|u_2(\tau)|^{\alpha-1})
|u_1(\tau)-u_2(\tau)|
\|_{\dot{K}^{\alpha s}_{q/\alpha,r_0}}
\,{\rm d}\tau\\
&\lesssim
\int_0^t
(t-\tau)^{\delta-1}
\left(
\max_{i=1,2}
\|u_i\|_{L^\infty(0,T\,;\,\dot{K}^s_{q,\alpha r_0})}
\right)^{\alpha-1}
\|u_1(\tau)-u_2(\tau)\|_{\dot{K}^s_{q,\alpha r_0}}
\,{\rm d}\tau,
\end{align*}
where
\[
r_0\ge\dfrac r\alpha,
\quad
\delta
\equiv
1-
\frac n2
\left[
\left(\frac\sigma n+\frac\alpha q\right)
-
\left(\frac sn+\frac1q\right)
\right]
>0.
\]
Then, by Gronwall's inequality,
\[
\|u_1(t)-u_2(t)\|_{\dot{K}^s_{q,r}}
=0.
\]
Namely, $u_1(t)=u_2(t)$.

\subsection{Proof of Theorem \ref{main2}}

To give a proof of Theorem \ref{main2}, we use the following lemma.

\begin{lemma}\label{lem1}
Let $s,\tilde{s}\in{\mathbb R}$, $1\le q,\tilde{q}\le\infty$ and $0<r,\tilde{r}\le\infty$, and let
\[
\beta
\equiv
\frac n2
\left[
\left(
\frac sn+\frac1q
\right)
-
\left(
\frac{\tilde{s}}n+\frac1{\tilde{q}}
\right)
\right]
>0.
\]
Assume that
\[
s\ge\tilde{s},
\quad
0\le
\frac{\tilde{s}}n+\frac1{\tilde{q}}
\le
\frac sn+\frac1q
\le1,
\quad
\begin{cases}
r=\infty &
\text{if
$\displaystyle
\frac sn+\frac1q=0
$},
\vspace{5pt} \\
r_0=\infty &
\text{if
$\displaystyle
\frac{\tilde{s}}n+\frac1{\tilde{q}}=0
$}.
\end{cases}
\]
If either
\begin{itemize}
\item[{\rm (1)}] $s>\tilde{s}$,
\quad
$q,\tilde{q}<\infty$,
\quad
$r=\infty$,
\quad
$\displaystyle
\frac{\tilde{s}}n+\frac1{\tilde{q}}
<
\frac sn+\frac1q
<1
$,
\quad or

\item[{\rm (2)}] $q,\tilde{q}<\infty$,
\quad
$r\le\tilde{r}$,
\quad
$\displaystyle
\frac sn+\frac1q
<1
$,
\quad
or

\item[{\rm (3)}] $r\le\min(1,\tilde{r})$,
\end{itemize}
then, for any
$f\in\dot{\mathcal K}^s_{q,r}({\mathbb R}^n)$,
\[
\lim_{t\downarrow0}
t^\beta
\|e^{t\Delta}f\|_{\dot{K}^{\tilde{s}}_{\tilde{q},\tilde{r}}}
=0.
\]
\end{lemma}

\begin{proof}
Let $f\in\dot{\mathcal K}^s_{q,r}({\mathbb R}^n)$ and $\varepsilon>0$.
Then there exists $g\in C_{\rm c}^\infty({\mathbb R}^n)$ such that
\[
\|f-g\|_{\dot{K}^s_{q,r}}
<
\varepsilon.
\]
By Lemma \ref{prop:LE-Herz2},
\begin{align*}
\|e^{t\Delta}f\|_{\dot{K}^{\tilde{s}}_{\tilde{q},\tilde{r}}}
&\lesssim
\|
e^{t\Delta}[f-g]
\|_{\dot{K}^{\tilde{s}}_{\tilde{q},\tilde{r}}}
+
\|e^{t\Delta}g\|_{\dot{K}^{\tilde{s}}_{\tilde{q},\tilde{r}}}\\
&\lesssim
t^{-\beta}
\|f-g\|_{\dot{K}^s_{q,r}}
+
\|g\|_{\dot{K}^{\tilde{s}}_{\tilde{q},\tilde{r}}}.
\end{align*}
Consequently,
\[
\lim_{t\downarrow0}
t^\beta
\|e^{t\Delta}f\|_{\dot{K}^{\tilde{s}}_{\tilde{q},\tilde{r}}}
<
\varepsilon.
\]
\end{proof}

Additionally, we use the continuity of $e^{t\Delta}$ at $t=0$ in Herz spaces.

\begin{proposition}\label{prop:time0-conti}
Let $s\in{\mathbb R}$, $1\le q\le\infty$ and $0<r\le\infty$ satisfy the either
\begin{center}
\lq\lq $q<\infty$''
\quad or \quad
\lq\lq $q=\infty$ and $r\le1$''.
\end{center}

If $f\in\dot{\mathcal K}^s_{q,r}({\mathbb R}^n)$, then,
\[
\lim_{t\downarrow0}
e^{t\Delta}f
=
f
\quad
\text{
in
$\dot{K}^s_{q,r}({\mathbb R}^n)$.
}
\]
\end{proposition}

\begin{proof}
We use Proposition \ref{prop:LE-Herz2} (2) and (3) as
\[
\|e^{t\Delta}f\|_{\dot{K}^s_{q,r}}
\le
C
\|f\|_{\dot{K}^s_{q,r}}
\]
for some $C>0$.

First, we assume that $f\in C_{\rm c}^\infty({\mathbb R}^n)$.
Using the fact
\[
e^{t\Delta}f-f
=
-
\int_0^t
e^{\tau\Delta}[\Delta f]
\,{\rm d}\tau,
\]
we have
\begin{align*}
\|e^{t\Delta}f-f\|_{\dot{K}^s_{q,r}}
=
\left\|
\int_0^t
e^{\tau\Delta}[\Delta f]
\,{\rm d}\tau
\right\|_{\dot{K}^s_{q,r}}
\le
Ct
\|\Delta f\|_{\dot{K}^s_{q,r}}
\to 0,
\end{align*}
as $t\downarrow0$.

Next, let $f\in\dot{\mathcal K}^s_{q,r}({\mathbb R}^n)$ and fix $\varepsilon>0$.
Then there exists $g\in C_{\rm c}^\infty({\mathbb R}^n)$ such that
\[
\|f-g\|_{\dot{K}^s_{q,r}}
<
\varepsilon.
\]
Moreover, by the previous case, there exists $T_0>0$ such that for all $t\in(0,T_0)$,
\[
\|e^{t\Delta}g-g\|_{\dot{K}^s_{q,r}}
<
\varepsilon.
\]
Hence,
\begin{align*}
\|e^{t\Delta}f-f\|_{\dot{K}^s_{q,r}}
&\le
\|e^{t\Delta}f-e^{t\Delta}g\|_{\dot{K}^s_{q,r}}
+
\|e^{t\Delta}g-g\|_{\dot{K}^s_{q,r}}
+
\|g-f\|_{\dot{K}^s_{q,r}}\\
&<
(C+2)\varepsilon.
\end{align*}
This is the desired result.
\end{proof}

We start the proof of Theorem \ref{main2}.
Let $u_1$ and $u_2$ are solutions of \eqref{eq:HH} with the same initial data $u_0\in\dot{K}^s_{q,r_0}({\mathbb R}^n)$.

At first, we give the following nonlinear estimate for the difference of the Duhamel term.
We decompose
\begin{align*}
|Nu_1(t)-Nu_2(t)|
&\lesssim
\int_0^t
e^{(t-\tau)\Delta}
\left[
|\cdot|^\gamma
\left(
|u_1(\tau)|^{\alpha-1}+|u_2(\tau)|^{\alpha-1}
\right)
|u_1(\tau)-u_2(\tau)|
\right]
\,{\rm d}\tau\\
&\lesssim
\int_0^t
e^{(t-\tau)\Delta}
\left[
|\cdot|^\gamma
|u_1(\tau)-e^{\tau\Delta}u_0|^{\alpha-1}
|u_1(\tau)-u_2(\tau)|
\right]
\,{\rm d}\tau\\
&\qquad+
\int_0^t
e^{(t-\tau)\Delta}
\left[
|\cdot|^\gamma
|u_2(\tau)-e^{\tau\Delta}u_0|^{\alpha-1}
|u_1(\tau)-u_2(\tau)|
\right]
\,{\rm d}\tau\\
&\qquad+
\int_0^t
e^{(t-\tau)\Delta}
\left[
|\cdot|^\gamma
|e^{\tau\Delta}u_0|^{\alpha-1}
|u_1(\tau)-u_2(\tau)|
\right]
\,{\rm d}\tau\\
&=:
{\rm I}(t)+{\rm II}(t)+{\rm III}(t).
\end{align*}

On the estimate of ${\rm I}(t)$, by Proposition \ref{prop:Meyer-Herz} and Proposition \ref{prop:Herz-fan} (2),
\begin{align*}
\|{\rm I}(t)\|_{\dot{K}^s_{q,\infty}}
&\lesssim
\sup_{0<\tau<t}
\left\|
|\cdot|^\gamma
|u_1(\tau)-e^{\tau\Delta}u_0|^{\alpha-1}
|u_1(\tau)-u_2(\tau)|
\right\|_{\dot{K}^\sigma_{q/\alpha,r_0}}\\
&\le
\|u_1-e^{\tau\Delta}u_0\|_{
L^\infty(0,t\,;\,\dot{K}^s_{q,r_0(\alpha-1)})
}^{\alpha-1}
\|u_1-u_2\|_{
L^\infty(0,t\,;\,\dot{K}^s_{q,\infty})
}.
\end{align*}

Similarly, on the estimate of ${\rm II}(t)$, we have
\[
\|{\rm II}(t)\|_{\dot{K}^s_{q,\infty}}
\lesssim
\|u_2-e^{\tau\Delta}u_0\|_{
L^\infty(0,t\,;\,\dot{K}^s_{q,r_0(\alpha-1)})
}^{\alpha-1}
\|u_1-u_2\|_{
L^\infty(0,t\,;\,\dot{K}^s_{q,\infty})
}.
\]

Additionally, on the estimate of ${\rm III}(t)$, taking $\tilde{q}$ such that
\[
\frac sn+\frac1q-\frac2{n(\alpha-1)}
<
\frac sn+\frac1{\tilde{q}}
<
\frac1{q_c},
\]
we can also choose $p$ and $\beta$ such that
\[
\frac1p
=
\frac{\alpha-1}{\tilde{p}}+\frac1q,
\quad
\beta
=
\frac n2\left(\frac1q-\frac1{\tilde{q}}\right),
\quad
0<
\frac sn+\frac1q
\le
\frac\sigma n+\frac1p
<1,
\]
and
\[
-\frac n2
\left[
\left(
\frac\sigma n+\frac1p
\right)
-
\left(
\frac sn+\frac1q
\right)
\right]
>-1,
\quad
-(\alpha-1)\beta>-1.
\]
Then, by Proposition \ref{prop:LE-Herz2} and Proposition \ref{prop:Herz-fan} (2),
\begin{align*}
&\|{\rm III}(t)\|_{\dot{K}^s_{q,\infty}}\\
&\qquad \lesssim
\int_0^t
(t-\tau)^{
-\frac n2
\left[
\left(
\frac\sigma n+\frac1p
\right)
-
\left(
\frac sn+\frac1q
\right)
\right]
}
\left\|
|\cdot|^\gamma
|e^{\tau\Delta}u_0|^{\alpha-1}
|u_1(\tau)-u_2(\tau)|
\right\|_{\dot{K}^\sigma_{p,\infty}}
\,{\rm d}\tau\\
&\qquad \lesssim
\int_0^t
(t-\tau)^{
-\frac n2
\left[
\left(
\frac\sigma n+\frac1p
\right)
-
\left(
\frac sn+\frac1q
\right)
\right]
-
\frac{\sigma-s}2
}
\tau^{-(\alpha-1)\beta}
\,{\rm d}\tau\\
&\qquad\qquad \times
\left(
\sup_{0<\tau<t}\tau^\beta
\|e^{\tau\Delta}u_0\|_{\dot{K}^s_{\tilde{q},\infty}}
\right)^{\alpha-1}
\|u_1-u_2\|_{L^\infty(0,t;\dot{K}^s_{q,\infty})}\\
&\qquad \sim
t^\delta
\left(
\sup_{0<\tau<t}\tau^\beta
\|e^{\tau\Delta}u_0\|_{\dot{K}^s_{\tilde{q},\infty}}
\right)^{\alpha-1}
\|u_1-u_2\|_{L^\infty(0,t;\dot{K}^s_{q,\infty})}.
\end{align*}

Next, we consider the uniqueness of the solutions.
By the previous argument,
\begin{align*}
&\|u_1(t)-u_2(t)\|_{\dot{K}^s_{q,\infty}}
=
\|Nu_1(t)-Nu_2(t)\|_{\dot{K}^s_{q,\infty}}\\
&\qquad \lesssim
\left(
\max_{i=1,2}
\|u_i-e^{\tau\Delta}u_0\|_{
L^\infty(
0,t;
\dot{K}^s_{q,r_0(\alpha-1)}
)
}
+
\sup_{0<\tau<t}\tau^\beta
\|e^{\tau\Delta}u_0\|_{\dot{K}^s_{\tilde{q},\infty}}
\right)^{\alpha-1}\\
&\qquad \qquad \times
\|u_1-u_2\|_{L^\infty(0,t;\dot{K}^s_{q,\infty})}.
\end{align*}
Since
\begin{align*}
&\|u_i-e^{\tau\Delta}u_0\|_{
L^\infty(0,t\,;\,\dot{K}^s_{q,r_0(\alpha-1)})
}\\
&\le
\|u_i-u_0\|_{
L^\infty(0,t\,;\,\dot{K}^s_{q,r_0(\alpha-1)})
}
+
\|u_0-e^{\tau\Delta}u_0\|_{
L^\infty(0,t\,;\,\dot{K}^s_{q,r_0(\alpha-1)})
}
\to0
\end{align*}
as $t\to0$ by Proposition \ref{prop:time0-conti}, and, by Lemma \ref{lem1},
\[
\sup_{0<\tau<t}
\tau^\beta
\|
e^{\tau\Delta}u_0
\|_{\dot{K}^s_{\tilde{q},\infty}}
\to0
\]
as $t\to0$, there exists $t_0>0$ such that
\[
\left(
\max_{i=1,2}
\|u_i-e^{\tau\Delta}u_0\|_{
L^\infty(
0,t_0;
\dot{K}^s_{q,r_0(\alpha-1)}
)
}
+
\sup_{0<\tau<t_0}\tau^\beta
\|e^{\tau\Delta}u_0\|_{\dot{K}^s_{\tilde{q},\infty}}
\right)^{\alpha-1}
<1.
\]
Then, for all $t\in(0,t_0]$, $u_1(t)=u_2(t)$.
It follows that there exists $T_0\ge t_0$ such that for all $t\in(0,T_0)$, $u_1(t)=u_2(t)$.

Here, set
\[
T_0
\equiv
\sup\{
t_0>0
\,:\,
\text{
$u_1(t)=u_2(t)$
for all $t\in(0,t_0)$
}
\},
\]
and assume that $T_0<T$.
By the continuity of
\[
t\mapsto u_i(t),
\quad
i=1,2,
\]
we have $u_1(T_0)=u_2(T_0)$.
Then arguing similar to the previous argument as a initial data $u_1(T_0)=u_2(T_0)$, we see that for a sufficiently small number $\varepsilon>0$, $u_1(T_0+\varepsilon)=u_2(T_0+\varepsilon)$.
This contradicts to the maximality of $T_0$.
Therefore, $T_0=T$ holds, and we finish the proof of Theorem \ref{main2}.

\section{Comparison of our results with previous results}\label{s:rem}

In this section, to compare the results on \cite{Tayachi20,CITT24}, we consider the comparison between Herz spaces and Lorentz spaces.

\begin{definition}
Let $0<p<\infty$ and $0<r\le\infty$.
Define the Lorentz space $L^{p,r}({\mathbb R}^n)$ by the space of all measurable functions $f$ with the finite quasi-norm
\[
\|f\|_{L^{p,r}}
\equiv
\begin{cases}
\displaystyle
\left(
\int_0^\infty
\left[
t^{\frac1p}
f^\ast(t)
\right]^r
\,\frac{{\rm d}t}t
\right)^{\frac1r},
& r<\infty, \\
\displaystyle
\sup_{t>0}
t^{\frac1p}
f^\ast(t),
& r=\infty,
\end{cases}
\]
where $f^\ast(t)$ is a decreasing rearrangement of $f$ defined by
\[
f^\ast(t)
\equiv
\inf\{
\lambda>0
\,:\,
|\{x\in{\mathbb R}^n\,:\,|f(x)|>\lambda\}|
\le t
\}.
\]
Additionally, let $s\in{\mathbb R}$, and define the Lorentz space with power weight $L_s^{p,r}({\mathbb R}^n)$ by the space of all measurable functions $f$ with the finite quasi-norm
\[
\|f\|_{L_s^{p,r}}
\equiv
\||\cdot|^sf\|_{L^{p,r}}.
\]
\end{definition}

As is well known, the real interpolation spaces for the Lorentz spaces are given as follows.

\begin{theorem}
Let $0<\theta<1$, $0<p,p_0,p_1<\infty$ and $0<r,r_0,r_1\le\infty$.
If
\[
p_0\ne p_1,
\quad
\frac1p
=
\frac{1-\theta}{p_0}+\frac\theta{p_1},
\]
then,
\[
(
L^{p_0,r_0}({\mathbb R}^n),
L^{p_1,r_1}({\mathbb R}^n)
)_{\theta,r}
\cong
L^{p,r}({\mathbb R}^n).
\]
\end{theorem}

Then the embedding between Herz and Lorentz spaces (with power weight) is given as follows.

\begin{proposition}\label{prop:Herz-Lorentz}
Let $s,\tilde{s}\in{\mathbb R}$, $0<p<\infty$ and $0<q,r\le\infty$.
\begin{itemize}
\item[{\rm (1)}] If $s<\tilde{s}$ and $s/n+1/q=\tilde{s}/n+1/p$, then,
\[
\dot{K}^s_{q,r}({\mathbb R}^n)
\hookleftarrow
L_{\tilde{s}}^{p,r}({\mathbb R}^n).
\]

\item[{\rm (2)}] If $s<\tilde{s}$ and $s/n+1/q=\tilde{s}/n+1/p$, then,
\[
\dot{K}^s_{q,r}({\mathbb R}^n)
\hookrightarrow
L_{\tilde{s}}^{p,r}({\mathbb R}^n).
\]
\end{itemize}
\end{proposition}

\begin{proof}
Since
\[
\|f\|_{\dot{K}^s_{q,r}}
\sim
\||\cdot|^{\tilde{s}}f\|_{\dot{K}^{s-\tilde{s}}_{q,r}},
\quad
\|f\|_{L_{\tilde{s}}^{p,r}}
=
\||\cdot|^{\tilde{s}}f\|_{L^{p,r}},
\]
we may the only case $\tilde{s}=0$.

(1) Let $\tilde{p}_i=-n/s_i$ for each $i=0,1$.
Then, by H\"older's inequality,
\begin{align*}
\|f\|_{\dot{K}^{s_i}_{q,q}}
\sim
\||\cdot|^{s_i}f\|_{L^q}
\lesssim
\||\cdot|^{s_i}\|_{L^{\tilde{p}_i,\infty}}
\|f\|_{L^{p_i,q}}
\sim
\|f\|_{L^{p_i,q}}
\end{align*}
for each $i=0,1$.
Then, combining with the real interpolations for the Herz spaces (Theorem \ref{thm:real-Herz}) and for the Lorentz spaces (see \cite{BeLo12}),
we obtain
\[
\|f\|_{\dot{K}^s_{q,r}}
\lesssim
\|f\|_{L^{p,r}}
\]
for all $f\in L^{p,r}({\mathbb R}^n)$.

(2) Let $\tilde{p}_i=n/s_i$ for each $i=0,1$.
Then, by H\"older's inequality,
\begin{align*}
\|f\|_{L^{p_i,q}}
\lesssim
\||\cdot|^{-s_i}\|_{L^{\tilde{p}_i,\infty}}
\||\cdot|^{s_i}f\|_{L^q}
\sim
\||\cdot|^{s_i}f\|_{L^q}
\sim
\|f\|_{\dot{K}^{s_i}_{q,q}}
\end{align*}
for each $i=0,1$.
Then, combining with the real interpolations for the Herz spaces (Theorem \ref{thm:real-Herz}) and for the Lorentz spaces (see \cite{BeLo12}),
we obtain
\[
\|f\|_{L^{p,r}}
\lesssim
\|f\|_{\dot{K}^s_{q,r}}
\]
for all $f\in\dot{K}^s_{q,r}({\mathbb R}^n)$.
\end{proof}

As far as I know, the case $s=\tilde{s}$ remains unsolved.
Therefore, we confine ourselves here to presenting a concrete example that allows a partial comparison between the two function spaces $\dot{K}^s_{q,r}({\mathbb R}^n)$ and $L_s^{q,r}({\mathbb R}^n)$.

\begin{example}\label{ex:Lo-He}
Let $q<\infty$, and let
\[
E_\beta
=
\bigcup_{j=1}^\infty
\left[
(2^{j-1}+1){\bm e}_1
+
B\left(\frac1{\sqrt[n]{j}^\beta}\right)
\right],
\quad
\beta\ge0.
\]
We have
\[
|\cdot|^{-s}\chi_{E_\beta}
\in
\begin{cases}
\dot{K}^s_{q,r}({\mathbb R}^n)
& \text{if and only if
$
\begin{cases}
\beta>\dfrac qr
& \text{if $r<\infty$}, \\
\beta\ge0
& \text{if $r=\infty$}.
\end{cases}
$} \\
\dot{\mathcal K}^s_{q,\infty}({\mathbb R}^n)
& \text{if and only if $\beta>0$}, \\
L_s^{q,r}({\mathbb R}^n)
& \text{if and only if $\beta>1$}.
\end{cases}
\]
\end{example}

\appendix
\section{Proof of Lemma \ref{lem:Herz-ball}}

Since $A_j\subset B(2^j)$, we have
\[
\|f\|_{\dot{K}^s_{q,r}}
\le
\left(
\sum_{j\in{\mathbb Z}}
\left[
2^{js}
\|f\chi_{B(2^j)}\|_{L^q}
\right]^r
\right)^{\frac1r}
=:
\|f\|_{\dot{K}^s_{q,r}({\rm ball})}.
\]
Thus, we may show the opposite estimate.

When $q\ge r$, we estimate
\begin{align*}
\|f\|_{\dot{K}^s_{q,r}({\rm ball})}^r
&\le
\sum_{j\in{\mathbb Z}}
\left[
2^{js}
\|f\chi_{B(2^{j-1})}\|_{L^q}
\right]^r
+
\sum_{j\in{\mathbb Z}}
\left[
2^{js}
\|f\chi_{A_j}\|_{L^q}
\right]^r\\
&=
2^{sr}
\|f\|_{\dot{K}^s_{q,r}({\rm ball})}^r
+
\|f\|_{\dot{K}^s_{q,r}}^r,
\end{align*}
and then,
\[
\|f\|_{\dot{K}^s_{q,r}({\rm ball})}
\le
\left(\frac1{1-2^{sr}}\right)^{\frac1r}
\|f\|_{\dot{K}^s_{q,r}}.
\]

When $q<r$, we estimate
\begin{align*}
\|f\|_{\dot{K}^s_{q,r}({\rm ball})}^q
&\le
\left(
\sum_{j\in{\mathbb Z}}
\left[
2^{js}
\|f\chi_{B(2^{j-1})}\|_{L^q}
\right]^r
\right)^{\frac qr}
+
\left(
\sum_{j\in{\mathbb Z}}
\left[
2^{js}
\|f\chi_{A_j}\|_{L^q}
\right]^r
\right)^{\frac qr}\\
&=
2^{sq}
\|f\|_{\dot{K}^s_{q,r}({\rm ball})}^q
+
\|f\|_{\dot{K}^s_{q,r}}^q,
\end{align*}
and then,
\[
\|f\|_{\dot{K}^s_{q,r}({\rm ball})}
\le
\left(\frac1{1-2^{sq}}\right)^{\frac1q}
\|f\|_{\dot{K}^s_{q,r}},
\]
as desired.

\section{Proof of Proposition \ref{prop:well-defined}}

At first, we prove the case (1) of the if part.
We estimate
\begin{align*}
\int_{B(2^J)}|f(x)|\,{\rm d}x
&=
\sum_{j\le J}
\int_{A_j}|f(x)|\,{\rm d}x\\
&\lesssim
\sum_{j\le J}
2^{\frac{nj}{q'}}
\left(\int_{A_j}|f(x)|^q\,{\rm d}x\right)^{\frac1q}\\
&\le
\sum_{j\le J}
2^{-sj+\frac{nj}{q'}}
\|f\|_{\dot{K}^s_{q,\infty}}
\sim
\|f\|_{\dot{K}^s_{q,\infty}}.
\end{align*}
Then, by Proposition \ref{prop:Herz-fan} (5),
\[
\dot{K}^s_{q,r}({\mathbb R}^n)
\hookrightarrow
L_{\rm loc}^1({\mathbb R}^n)
\]
for $0<r\le\infty$.

Next, we prove the case (2) of the if part.
We estimate
\begin{align*}
\int_{B(2^J)}|f(x)|\,{\rm d}x
&=
\sum_{j\le J}
\int_{A_j}|f(x)|\,{\rm d}x\\
&\lesssim
\sum_{j\le J}
2^{\frac{nj}{q'}}
\left(\int_{A_j}|f(x)|^q\,{\rm d}x\right)^{\frac1q}\\
&\le
\sup_{j\le J}
2^{-sj+\frac{nj}{q'}}
\|f\|_{\dot{K}^s_{q,1}}
=
\|f\|_{\dot{K}^s_{q,1}}.
\end{align*}
Then, by Proposition \ref{prop:Herz-fan} (5),
\[
\dot{K}^s_{q,r}({\mathbb R}^n)
\hookrightarrow
L_{\rm loc}^1({\mathbb R}^n)
\]
for $0<r\le1$.

Finally, we prove the only if part.
Let
\[
f(x)
=
|x|^{-\alpha}
\left(\log\frac1{|x|}\right)^{-\beta}
\chi_{B(2^{-10})},
\quad
\alpha,\beta\ge0.
\]
Then,
\begin{align*}
\|f\|_{\dot{K}^s_{q,\infty}}
&=
\sup_{j\le-10}
\left[
2^{js}
\|f\chi_{A_j}\|_{L^q}
\right]\\
&\sim
\sup_{j\le-10}
\left[
2^{js-\alpha j+\frac{nj}q}
\left(\log\frac1{2^j}\right)^{-\beta}
\right]
<\infty
\end{align*}
if and only if
\[
\alpha\le s+\frac nq.
\]
Moreover,
\begin{align*}
\|f\|_{\dot{K}^s_{q,r}}^r
&=
\sum_{j\le-10}
\left[
2^{js}
\|f\chi_{A_j}\|_{L^q}
\right]^r\\
&\sim
\sum_{j\le-10}
\left[
2^{js-\alpha j+\frac{nj}q}
\left(\log\frac1{2^j}\right)^{-\beta}
\right]^r
<\infty
\end{align*}
if and only if either
\[
\alpha<s+\frac nq
\]
or
\[
\alpha=s+\frac nq,
\quad
\beta>\frac1r.
\]
Meanwhile, as is well known,
\[
f\in L^1({\mathbb R}^n)
\]
if and only if either
\[
\alpha<n
\]
or
\[
\alpha=n,
\quad
\beta>1.
\]

Therefore, assuming either
\[
\text{\lq\lq
$\dfrac sn+\dfrac1q=1$
and
$r>1$
''}
\quad \text{or} \quad
\text{\lq\lq
$\frac sn+\dfrac1q>1$
''},
\]
we can choose $\alpha=n$ and $\beta=1$, and then,
\[
f\in
\dot{K}^s_{q,r}({\mathbb R}^n)
\setminus
L^1({\mathbb R}^n).
\]
This is the desired result.

\section{Proof of Proposition \ref{prop:dense}}

At first, we assume that $q,r<\infty$.
Let $f\in\dot{K}^s_{q,r}({\mathbb R}^n)$ and $\varepsilon>0$, and let $J\in{\mathbb N}$ satisfy
\[
\left(
\sum_{j\in{\mathbb Z}\setminus[-J,J]}
\left[
2^{js}
\|f\chi_{A_j}\|_{L^q}
\right]^r
\right)^{\frac1r}
<
\varepsilon.
\]
For each $j\in{\mathbb Z}$, there exists $g_j\in C_{\rm c}^\infty({\mathbb R}^n)$ such that
\[
{\rm supp}(g_j)\subset A_j,
\quad
\|(f-g_j)\chi_{A_j}\|_{L^q}
<
2^{-(js)^2}\varepsilon.
\]
Then, setting
\[
G_J=
\sum_{j=-J}^J
g_j,
\]
we obtain
\begin{align*}
\|f-G_J\|_{\dot{K}^s_{q,r}}^r
&=
\sum_{j=-J}^J
\left[
2^{js}
\|(f-g_j)\|_{L^q}
\right]^r
+
\sum_{j\in{\mathbb Z}\setminus[-J,J]}
\left[
2^{js}
\|f\chi_{A_j}\|_{L^q}
\right]^r\\
&<
\sum_{j\in{\mathbb Z}}
\left[
2^{js-(js)^2}
\varepsilon
\right]^r
+
\varepsilon^r
\sim
\sum_{j\in{\mathbb Z}}
\left[
2^{-\left(js-\frac12\right)^2}
\varepsilon
\right]^r
+
\varepsilon^r
\sim
\varepsilon^r,
\end{align*}
that is, $C_{\rm c}^\infty({\mathbb R}^n)$ is dense in $\dot{K}^s_{q,r}({\mathbb R}^n)$.

Next, in the case $q<\infty$ and $r=\infty$, according to Example \ref{ex:Lo-He},
\[
|\cdot|^{-s}
\chi_{E_0}
\in
\dot{K}^s_{q,\infty}({\mathbb R}^n)
\setminus
\dot{\mathcal K}^s_{q,\infty}
({\mathbb R}^n).
\]
Hence, $C_{\rm c}^\infty({\mathbb R}^n)$ is not dense in $\dot{K}^s_{q,\infty}({\mathbb R}^n)$.

Finally, when $q=\infty$, we verify the claim
\[
\inf_{g\in C_{\rm c}^\infty}
\|\chi_{A_1}-g\|_{\dot{K}^s_{\infty,r}}
\ge
\frac{\min(1,2^s)}2.
\]
To show this, we consider the following 2 parts:
\begin{center}
(i) $|A_1\setminus{\rm supp}(g)|>0$,
\quad
(ii) $|A_1\setminus{\rm supp}(g)|=0$.
\end{center}

\begin{itemize}
\item[(i)] Assume that $|A_1\setminus{\rm supp}(g)|>0$.
Then
\begin{align*}
\|\chi_{A_1}-g\|_{\dot{K}^s_{\infty,r}}
\ge
2^s\|(1-g)\chi_{A_1\setminus{\rm supp}(g)}\|_{L^\infty}
=
2^s.
\end{align*}

\item[(ii)] Assume that $|A_1\setminus{\rm supp}(g)|=0$.
Then
\begin{align*}
\|\chi_{A_1}-g\|_{\dot{K}^s_{\infty,r}}
&\ge
\|\chi_{A_1}-g\|_{\dot{K}^s_{\infty,\infty}}\\
&\ge
\max\left(
\|g\chi_{A_0}\|_{L^\infty},
2^s
\|(1-g)\chi_{A_1}\|_{L^\infty}
\right)\\
&\ge
\max\left(
\sup_{x\in S^{n-1}}|g(x)|,
2^s
\left|
1-\sup_{x\in S^{n-1}}|g(x)|
\right|
\right)\\
&\ge
\min(1,2^s)
\max(\rho_0,1-\rho_0)
\ge
\frac{\min(1,2^s)}2,
\end{align*}
where $\rho_0=\sup\limits_{x\in S^{n-1}}|g(x)|$.
\end{itemize}

Therefore, the claim is proved, and it follows that $C_{\rm c}^\infty({\mathbb R}^n)$ is not dense in $\dot{K}^s_{\infty,r}({\mathbb R}^n)$.

\section{Proof of Proposition \ref{prop:Gauss}}

We prove only the case $r<\infty$.
We estimate
\begin{align*}
\|e^{-|\cdot|^2}\|_{\dot{K}^s_{q,r}}^r
&\sim
\sum_{j\in{\mathbb Z}}
\left[
2^{js}
\|e^{-4^j}\chi_{A_j}\|_{L^q}
\right]^r
\sim
\sum_{j\in{\mathbb Z}}
\left[
2^{js+\frac{nj}q}
e^{-4^j}
\right]^r\\
&=
\sum_{j=-\infty}^{-1}
+
\sum_{j=0}^\infty
=:
{\bf I}+{\bf II}.
\end{align*}

On the term ${\bf I}$, since
\[
\frac12
\le
e^{-4^j}
\le
1
\]
for all $j\le-1$, the necessary and sufficient condition of
\[
{\bf I}
\sim
\sum_{j=-\infty}^{-1}
\left[
2^{js+\frac{nj}q}
\right]^r
<
\infty
\]
is $s+n/q>0$.
On the term ${\bf II}$, replacing $k=2^j$, we have
\begin{align*}
{\bf II}
\le
\sum_{k=1}^\infty
\left[
k^{s+\frac nq}
e^{-k^2}
\right]^r
\lesssim
1.
\end{align*}
Then we finish the proof of Proposition \ref{prop:Gauss}.

\section{Proof of Theorem \ref{thm:real-Herz}}\label{App:interpolation}

Let $s\in{\mathbb R}$ and $0<r\le\infty$, let $X$ be a quasi-Banach space, and define the vector-valued space $\ell_s^r(X)$ by the space of all $F=\{f_j\}_{j\in{\mathbb Z}}\subset X$ with the finite quasi-norm
\[
\|F\|_{\ell_s^r(X)}
\equiv
\left(
\sum_{j\in{\mathbb Z}}
\left[
2^{js}\|f_j\|_X
\right]^r
\right)^{\frac1r}.
\]
Then the real interpolation spaces of these vector valued-spaces are given as follows.

\begin{theorem}\label{thm:real-vector}
Let $0<\theta<1$, $s,s_0,s_1\in{\mathbb R}$ and $0<r,r_0,r_1\le\infty$, and let $X$ be a quasi-Banach space.
If
\[
s_0\ne s_1,
\quad
s=(1-\theta)s_0+\theta s_1,
\]
then,
\[
(
\ell_{s_0}^{r_0}(X),
\ell_{s_1}^{r_1}(X)
)_{\theta,r}
\cong
\ell_s^r(X).
\]
\end{theorem}

Although this theorem is stated for the case $1\le r,r_0,r_1\le\infty$ (see \cite[Theorem 5.6.1 (1)]{BeLo12}), the case including $0<r,r_0,r_1<1$ can also be proved by a similar argument.

Additionally, to prove Theorem \ref{thm:real-Herz}, we use the following results of the interpolation theory.

\begin{theorem}[{\cite[Theorem 3.7.1]{BeLo12}}]\label{thm:real-dual}
Let $0<\theta<1$ and $1\le r<\infty$, and let $X_0$ and $X_1$ be quasi-Banach spaces.
Then,
\[
(X_0,X_1)_{\theta,r}^\ast
\cong
(X_0^\ast,X_1^\ast)_{\theta,r'}.
\]
\end{theorem}

\begin{definition}
Let $X$ and $Y$ be quasi-Banach spaces.
A bounded linear operator $R:X\to Y$ is said to be a retraction if there exists a bounded linear operator $S:Y\to X$ such that
\[
RS
=
{\rm id}_Y.
\]
We call $S$ as a coretraction and $Y$ is a retreat of $X$.
\end{definition}

\begin{theorem}[{\cite[Theorem 6.4.2]{BeLo12}}]\label{thm:real-theory}
Let $(X_0,X_1)$ and $(Y_0,Y_1)$ be interpolation couples.
Suppose that $S$ is coretraction of $(Y_i,X_i)$, $i=0,1$ and $R$ is a retraction of $(X_i,Y_i)$, $i=0,1$.
Then, $S$ is an isomorphic mapping from $(Y_0,Y_1)_{\theta,r}$ onto a complemented subspace of $(X_0,X_1)_{\theta,r}$.
\end{theorem}

Here, we start the proof of Theorem \ref{thm:real-Herz}.

\noindent
\underline{\bf Step 1}.
Write
\[
Sf
=
\{(Sf)_j\}_{j\in{\mathbb Z}}
=
\{f\chi_{A_j}\}_{j\in{\mathbb Z}},
\quad
R[\{f_j\}_{j\in{\mathbb Z}}]
=
\sum_{j\in{\mathbb Z}}
f_j
\chi_{A_j}.
\]
Note that
\begin{align*}
\|f\|_{\dot{K}^s_{q,r}}^r
=
\sum_{j\in{\mathbb Z}}
\left[
2^{js}
\|f\chi_{A_j}\|_{L^q}
\right]^r
=
\sum_{j\in{\mathbb Z}}
\left[
2^{js}
\|(Sf)_j\|_{L^q}
\right]^r.
\end{align*}
Then
$S:
\dot{K}^s_{q,r}({\mathbb R}^n)
\to
\ell_s^r({\mathbb Z}\,;L^q({\mathbb R}^n))
$
is an isometry.
Meanwhile, since
\begin{align*}
\|RF\|_{\dot{K}^s_{q,r}}^r
=
\sum_{k\in{\mathbb Z}}
\left[
2^{js}
\|(RF)\chi_{A_j}\|_{L^q}
\right]^r
=
\sum_{k\in{\mathbb Z}}
\left[
2^{js}
\|f_k\chi_{A_j}\|_{L^q}
\right]^r
\end{align*}
for any
$F=
\{f_j\}_{j\in{\mathbb Z}}\in
\ell_s^r({\mathbb Z}\,;L^q({\mathbb R}^n))
$,
we see that
$\dot{K}^s_{q,r}({\mathbb R}^n)$ is a retreat of
$
\ell_s^r({\mathbb Z}\,;L^q({\mathbb R}^n))
$.
Therefore, it follows from Theorem \ref{thm:real-theory} that $S$ is an isomorphic from
$(
\dot{K}^{s_0}_{q,r_0}({\mathbb R}^n),
\dot{K}^{s_1}_{q,r_1}({\mathbb R}^n)
)_{\theta,r}$
to a subspace of
$(
\ell_{s_0}^{r_0}
({\mathbb Z}\,;L^q({\mathbb R}^n)),
\ell_{s_1}^{r_1}
({\mathbb Z}\,;L^q({\mathbb R}^n))
)_{\theta,r}$.
Hence, by Theorem \ref{thm:real-vector},
\begin{align*}
\|f\|_{(
\dot{K}^{s_0}_{q,r_0},
\dot{K}^{s_1}_{q,r_1}
)_{\theta,r}}
\sim
\|Sf\|_{(
\ell_{s_0}^{r_0}(L^q),
\ell_{s_1}^{r_1}(L^q)
)_{\theta,r}}
\sim
\|Sf\|_{\ell_s^r(L^q)}
=
\|f\|_{\dot{K}^s_{q,r}},
\end{align*}
and then,
\[
(
\dot{K}^{s_0}_{q,r_0}({\mathbb R}^n),
\dot{K}^{s_1}_{q,r_1}({\mathbb R}^n)
)_{\theta,r}
\hookrightarrow
\dot{K}^s_{q,r}({\mathbb R}^n).
\]

\noindent
\underline{\bf Step 2}.
Assume that $1<q,r\le\infty$.
Then, by the duality argument (see Theorem \ref{thm:real-dual} and Proposition \ref{prop:Herz-fan} (8)),
\begin{align*}
\dot{K}^s_{q,r}({\mathbb R}^n)
&\cong
\left(
\dot{K}^{-s}_{q',r'}({\mathbb R}^n)
\right)^\ast\\
&\hookrightarrow
(
\dot{K}^{-s_0}_{q',r_0'}({\mathbb R}^n),
\dot{K}^{-s_1}_{q',r_1'}({\mathbb R}^n)
)_{\theta,r'}^\ast
\cong
(
\dot{K}^{s_0}_{q,r_0}({\mathbb R}^n),
\dot{K}^{s_1}_{q,r_1}({\mathbb R}^n)
)_{\theta,r}.
\end{align*}

\noindent
\underline{\bf Step 3}.
Assume that $q\le1$ or $r\le1$, and let $f\in\dot{K}^s_{q,r}({\mathbb R}^n)$.
Remark that there exists $\varepsilon=\varepsilon(q,r)>0$ such that
\[
\frac q\varepsilon>1,
\quad
\frac r\varepsilon>1,
\quad
|f|^\varepsilon
\in
\dot{K}^{s\varepsilon}_{
\frac q\varepsilon,\frac r\varepsilon
}({\mathbb R}^n).
\]
Then,
\[
\||f|^\varepsilon\|_{\dot{K}^{s\varepsilon}_{
q/\varepsilon,r/\varepsilon
}}
\sim
\||f|^\varepsilon\|_{(
\dot{K}^{s_0\varepsilon}_{
q/\varepsilon,r_0/\varepsilon
},
\dot{K}^{s_1\varepsilon}_{
q/\varepsilon,r_1/\varepsilon
}
)_{\theta,r/\varepsilon}}.
\]
Since
$
|f|^\varepsilon
=
f_0+f_1
\in
\dot{K}^{s_0\varepsilon}_{
q/\varepsilon,r_0/\varepsilon
}({\mathbb R}^n)
+
\dot{K}^{s_1\varepsilon}_{
q/\varepsilon,r_1/\varepsilon
}({\mathbb R}^n)
$
implies
\[
|f|
=
(f_0+f_1)^{\frac1\varepsilon}
\lesssim
|f_0|^{\frac1\varepsilon}
+
|f_1|^{\frac1\varepsilon},
\]
we have
\begin{align*}
K(t,|f|^\varepsilon)
&=
\inf_{
|f|^\varepsilon=f_0+f_1
\in
\dot{K}^{s_0\varepsilon}_{
q/\varepsilon,r_0/\varepsilon
}
+
\dot{K}^{s_1\varepsilon}_{
q/\varepsilon,r_1/\varepsilon
}
}
\left[
\|f_0\|_{\dot{K}^{s_0\varepsilon}_{
q/\varepsilon,r_0/\varepsilon
}}
+t
\|f_1\|_{\dot{K}^{s_1\varepsilon}_{
q/\varepsilon,r_1/\varepsilon
}}
\right]\\
&=
\inf_{
|f|^\varepsilon=f_0+f_1
\in
\dot{K}^{s_0\varepsilon}_{
q/\varepsilon,r_0/\varepsilon
}
+
\dot{K}^{s_1\varepsilon}_{
q/\varepsilon,r_1/\varepsilon
}
}
\left[
\left\||f_0|^{\frac1\varepsilon}\right\|_{
\dot{K}^{s_0}_{q,r_0}
}^\varepsilon
+t
\left\||f_1|^{\frac1\varepsilon}\right\|_{
\dot{K}^{s_1}_{q,r_1}
}^\varepsilon
\right]\\
&\gtrsim
K\left(
t^{\frac1\varepsilon},f
\right)^\varepsilon.
\end{align*}
Consequently,
\begin{align*}
\|f\|_{\dot{K}^s_{q,r}}^r
&=
\||f|^\varepsilon\|_{\dot{K}^{s\varepsilon}_{
q/\varepsilon,r/\varepsilon
}}^{\frac r\varepsilon}
\sim
\||f|^\varepsilon\|_{(
\dot{K}^{s_0\varepsilon}_{
q/\varepsilon,r_0/\varepsilon
},
\dot{K}^{s_1\varepsilon}_{
q/\varepsilon,r_1/\varepsilon
}
)_{\theta,r/\varepsilon}
}^{\frac r\varepsilon}\\
&\gtrsim
\int_0^\infty
\left[
t^{-\theta}
K\left(
t^{\frac1\varepsilon},
f
\right)^\varepsilon
\right]^{\frac r\varepsilon}
\,\frac{{\rm d}t}t
\sim
\int_0^\infty
\left[
s^{-\theta\varepsilon}
K(s,f)^\varepsilon
\right]^{\frac r\varepsilon}
\,\frac{{\rm d}s}s\\
&=
\|f\|_{(
\dot{K}^{s_0}_{q,r_0},
\dot{K}^{s_1}_{q,r_1}
)_{\theta,r}}^r.
\end{align*}

{\bf Acknowledgements.}
The authors are grateful to Professor Taniguchi for pointing out an error in the assumptions of the main result and for his helpful comments, which enabled us to correct it.
The first-named author (N.H.) was supported by the Grant-in-Aid for JSPS Fellows (No. 25KJ0222).
The second author (M.I.) is supported by JSPS, the Grant-in-Aid for Scientific Research (C) (No. 23K03174) and the Grant-in-Aid for Transformative Research Areas (B) (No. 25H01453).

\bibliographystyle{amsplain}

\end{document}